\documentclass[12pt]{amsart}
\usepackage[dvips]{graphicx}
\usepackage{amsmath,graphics}
\usepackage{amscd}
\usepackage{amsfonts,amssymb,xypic}
\theoremstyle{plain}

\newtheorem{thm}{Theorem}%[section]

\newtheorem{lemma}[thm]{Lemma}

\newtheorem{prop}[thm]{Proposition}

 \theoremstyle{remark}
 \newtheorem{remark}[thm]{Remark}
\newcommand{\al}{\alpha}
\newcommand{\be}{\beta}

\newcommand{\ga}{\gamma}
\newcommand{\ep}{\varepsilon}

\renewcommand{\th}{\theta}
\newcommand{\la}{\lambda}
\newcommand{\om}{\omega}
\newcommand{\si}{\sigma}

\newcommand{\Si}{\Sigma}

\newcommand{\ZZ}{{\mathbb Z}}

\newcommand{\CC}{{\mathbb C}}

\newcommand{\SL}{{\rm SL}}
\newcommand{\GL}{{\rm GL}}

\newcommand{\Aut}{\operatorname{Aut}}

\def\co{\colon\thinspace}

\begin{document}
\title[Metabelian SL$(n,\CC)$ representations of knot groups II]{Metabelian SL$(n,\CC)$ representations of knot groups II:  fixed points}

\author{Hans U. Boden}
\address{Department of Mathematics, McMaster University,
Hamilton, Ontario} \email{boden@mcmaster.ca}
\thanks{The first named author was supported by a grant from the Natural Sciences and Engineering Research Council of Canada.}

\author{Stefan Friedl}
\address{
University of Warwick, Coventry, UK}
\email{sfriedl@gmail.com}


\date{\today}
\begin{abstract}
Given a knot $K$ in an integral homology sphere $\Si$ with exterior $N_K$,
there is a natural action of the cyclic group $\ZZ/n$ on the space of $\SL(n,\CC)$ representations
of the knot group $\pi_1(N_K)$, and this induces an action on the $\SL(n,\CC)$ character variety.
We identify the fixed points of this action in terms of characters of metabelian representations,
and we apply this to show that the twisted Alexander polynomial $\Delta^\al_{K,1}(t)$
associated to an irreducible metabelian $\SL(n,\CC)$
representation $\al$ is actually a polynomial in $t^n$.

%We then show that for any irreducible metabelian representation $\al$, we have
%$ \dim_{\CC} H^1(N_K;sl(n,\CC)_{\ad \al}) \geq n-1$. If equality holds,
%then we prove that the  character
%$\xi_\al$ is a smooth point in $X_n(N_K)$ and that there exists an $(n-1)$--dimensional family of
%characters of irreducible $\SL(n,\CC)$ representations of $\pi_1(N_K)$ near $\xi_\al$.
%We relate the cohomological condition that $\dim_{\CC} H^1(N_K;sl(n,\CC)_{\ad \al}) = n-1$ to the vanishing of $H_1(\widehat{\Si}_\varphi)$,
%where $\widehat{\Si}_\varphi$ denotes the associated metabelian branched cover
%of $\Si$ branched along $K$.
 \end{abstract}
\maketitle

%==========================================================
\section{Introduction}

Suppose $K$ is a knot. Throughout this paper we will always understand this to mean that $K$ is  an oriented simple closed curve in an integral homology 3-sphere $\Si$.
We write $N_K=\Si^3\smallsetminus \tau(K),$ where $\tau(K)$ denotes an open tubular neighborhood of $K$.

The study of metabelian representations and metabelian quotients of knot groups goes back to the pioneering work of
Neuwirth \cite{Ne65},  de Rham \cite{dRh68}, Burde \cite{Bu67}  and Fox \cite{Fo70}
(see also \cite[Section~14]{BZ03}).
The theory was further developed
by many authors, including Hartley \cite{Ha79,Ha83},  Livingston \cite{Li95}, Letsche \cite{Le00}, Lin \cite{Lin01}, Nagasato
\cite{Na07} and  Jebali \cite{Je08}.
In \cite{BF08} we proved a classification theorem for irreducible metabelian representations,
%and the  recent results of Abdelghani, Heusener, and Jebali \cite{AHJ07}
%regarding the deformation of metabelian representations.
and in this paper we continue  our study of metabelian representations of knot groups.

%The aforementioned papers for the most part studied the case of two and three dimensional representations.

We begin by introducing some terminology.
Given a topological space $M$, let $R_n(M)$ be the space
 of $\SL(n,\CC)$ representations of $\pi_1(M)$   and
 $X_n(M)$ the associated character variety.
 We use $\xi_\al$ to denote the character of the representation
 $\al \co \pi_1(M) \to \SL(n,\CC)$. We will often make use of the important fact
 that two irreducible representations
 determine the same character if and only if they are conjugate (see \cite[Corollary 1.33]{LM85}).

Now suppose $K$ is a knot.
There is an action  of the group $\ZZ/n$ on the representation variety
$R_n(N_K)$ given by twisting by the $n$--th roots of unity $\om^k= e^{2 \pi ik/n} \in U(1)$.
(This is a special case of the more general twisting operation described in \cite[Ch. 5]{LM85}.)
More precisely, we write $\ZZ/n = \langle \si \mid \si^n=1\rangle$ and
set
$(\si \cdot \al)(g) = \om^{\ep(g)} \al(g)$ for each  $g \in \pi_1(N_K)$,
where $\ep \co \pi_1(N_K) \to H_1(N_K) = \ZZ$ is determined by the given orientation of the knot.

This constructs an  action of $\ZZ/n$ on $R_n(N_K)$
which, in turn, descends to an action
on the character variety $X_n(N_K)$.
Our main result identifies the fixed points of
$\ZZ/n$ in $X_n^*(N_K)$, the irreducible characters, as those associated to
metabelian representations.

\begin{thm}\label{thm1}
The character $\xi_\al$ of an irreducible representation  $\al \co \pi_1(N_K) \to \SL(n,\CC)$ is
fixed under the $\ZZ/n$ action if and only if $\al$ is metabelian.
\end{thm}

In proving this result, we actually characterize the entire fixed point set  $X_n(N_K)^{\ZZ/n}$
in terms of characters $\xi_\al$ of the metabelian representations $\al=\al_{(n,\chi)}$
described in Subsection \ref{sec2-3} (see  Theorem \ref{thm4}).
When $n=2,$ it turns out that every metabelian $\SL(2,\CC)$ representation is dihedral and in this case
Theorem \ref{thm1} was first proved by F. Nagasato and Y. Yamaguchi (cf. \cite[Proposition 4.8]{NY08}).

As an application of Theorem \ref{thm1}, we prove a
result about the twisted Alexander polynomials associated  to metabelian representations.
This result was first shown by C. Herald, P. Kirk and C. Livingston in \cite{HKL08} using completely
different methods. Our approach is elementary and quite natural, and it is explained in Section  \ref{section:twialex}, where we apply it to give an answer
to a question raised by Hirasawa and Murasugi in \cite{HM09}.

\bigskip \noindent
{\bf Acknowledgments. \ } The authors would like to thank
Steven Boyer, Christopher Herald, Michael Heusener, Paul Kirk, Charles Livingston, Andrew Nicas and Adam Sikora for
generously sharing their knowledge, wisdom, and insight.
%SF We are especially grateful to Michael Heusener for bringing the results of \cite{AHJ07} to our attention.
%\footnote{The paper is a little odd now. We thank Heusener for referring us to a theorem, which by now we no longer use!}
%\footnote{second change}
We would also like to thank Fumikazu Nagasato and Yoshikazu Yamaguchi for communicating
the results of their paper to us.
%HB was supported by a grant from the Natural Sciences and
%Engineering Research Council of Canada.

%==========================================================
\section{The classification of metabelian representations of knot groups}

In this section we recall some results from \cite{BF08}
regarding the classification of metabelian  representations of knot groups.
% and we provide some general results on the twisted homology and cohomology groups.

%==========================================================
\subsection{Preliminaries}
Given a group $\pi$, we shall write  $\pi^{(n)}$ for the $n$--th term of the
derived series of $\pi$. These subgroups are defined inductively by setting
$\pi^{(0)}=\pi$ and $\pi^{(i+1)}=[\pi^{(i)},\pi^{(i)}]$.
The group $\pi$ is called \emph{metabelian} if $\pi^{(2)}=\{e\}$.

Suppose $V$ is a finite dimensional vector space over $\CC$.
A  representation $\varrho \co \pi\to \Aut(V)$
is called \emph{metabelian}
 if $\varrho$ factors through $\pi/\pi^{(2)}$. The representation $\varrho$ is
 called \emph{reducible}
if there exists a proper subspace $U\subset V$ invariant under $\varrho(\ga)$ for all $\ga\in  \pi.$
Otherwise $\varrho$ is called \emph{irreducible} or \emph{simple}.
If $\varrho$ is the direct sum of simple representations, then $\varrho$ is called \emph{semisimple}.

Two representations
$\varrho_1 \co \pi \to \Aut(V)$ and $\varrho_2 \co \pi \to \Aut(W)$ are called \emph{isomorphic}
if there exists an isomorphism $\phi \co V\to W$ such that $\phi^{-1} \circ \varrho_1(g)\circ \phi=\varrho_2(g)$ for all $g\in \pi$.

%===========================================================
\subsection{Metabelian quotients of knot groups}\label{section:metabk}

Let $K\subset \Si^3$ be a knot in an integral homology 3-sphere.
In the following we denote by $\widetilde{N}_K$
the infinite cyclic cover of $N_K$ corresponding to the abelianization $\pi_1(N_K)\to H_1(N_K) \cong \ZZ$.
Therefore $\pi_1(\widetilde{N}_K)=\pi_1(N_K)^{(1)}$ and
$$ H_1(N_K;\ZZ[t^{\pm 1}])=H_1(\widetilde{N}_K) \cong \pi_1(N_K)^{(1)}/\pi_1(N_K)^{(2)}.$$
The $\ZZ[t^{\pm 1}]$--module structure is
given on the right hand side
by $t^n\cdot g:=\mu^{-n}g\mu^n$, where $\mu$ is a meridian
of $K$.

For a knot $K$, we set $\pi:=\pi_1(N_K)$ and consider the short exact sequence
$$ 1\to \pi^{(1)}/\pi^{(2)}\to \pi/\pi^{(2)}\to\pi/\pi^{(1)}\to 1. $$
Since $\pi/\pi^{(1)}=H_1(N_K)\cong \ZZ$, this sequence splits and we get isomorphisms
$$ \begin{array}{rcccl} \pi/\pi^{(2)}
&\cong & \pi/\pi^{(1)} \ltimes \pi^{(1)}/\pi^{(2)}
&\cong &\ZZ \ltimes \pi^{(1)}/\pi^{(2)}   \cong \ZZ \ltimes H_1(N_K;\ZZ[t^{\pm 1}]) \\
    g &\mapsto &(\mu^{\ep(g)},\mu^{-\ep(g)}g) &\mapsto &(\ep(g),\mu^{-\ep(g)}g), \end{array}  $$
where the semidirect products are taken with respect to the $\ZZ$ actions defined by
letting $n \in \ZZ$ act by conjugation by $\mu^n$ on $\pi^{(1)}/\pi^{(2)}$ and by multiplication
by $t^n$ on $H_1(N_K; \ZZ[t^{\pm1}])$.

%This demonstrates the following lemma.
%
%\begin{lem}\label{lem4}
%Given a knot $K$, the set of
%metabelian representations of $\pi_1(N_K)$ can be canonically identified with the set of
%representations of $\ZZ \ltimes H_1(N_K;\ZZ[t^{\pm 1}])$.
%\end{lem}

%===========================================================
\subsection{Irreducible metabelian $\SL(n, \CC)$ representations of knot groups}
\label{sec2-3}
Let $K$ be a knot. We write $H=H_1(N_K;\ZZ[t^{\pm 1}])$.
The discussion of the previous section shows that irreducible metabelian $\SL(n,\CC)$ representations of $\pi_1(N_K)$ correspond precisely to the irreducible  $\SL(n,\CC)$ representations of $\ZZ\ltimes H$.

%Throughout this paper we will frequently make use of the following well--known facts:
%\begin{enumerate}
%\item  $H$ is a finitely generated
%$\ZZ[t^{\pm 1}]$--module such that multiplication by $t-1$ is an isomorphism.
%\item Given $n\in \NN$ there exists a natural identification $H/(t^n-1)\cong H_1(L_n)$, where $L_n$ denotes the $n$--fold cover of $\Si^3$ branched over $K$.
%\end{enumerate}

Let
$\chi \co H\to \CC^*$ be a character which factors through $H/(t^n-1)$ and suppose $z\in S^1$ with $z^n=(-1)^{n+1}$. Then it follows from
\cite[Section~3]{BF08} that, for $(j, h) \in \ZZ\ltimes H,$ setting
$$  \al_{(\chi,z)} (j,h) =
 \begin{pmatrix}
 0& &\dots &z \\
 z&0&\dots &0 \\
\vdots &\ddots &\ddots&\vdots \\
     0&\dots &z &0 \end{pmatrix}^j
     \begin{pmatrix} \chi(h) &0&\dots &0 \\
 0&\chi(th) &\dots &0 \\
\vdots &&\ddots &\vdots \\ 0&0&\dots &\chi(t^{n-1}h) \end{pmatrix}
$$
defines an  $\SL(n,\CC)$ representation whose isomorphism type of this representation does not depend on the choice of $z$.
%Note that $\al_{(n,\chi,z)}$ factors through $\ZZ\ltimes H/(t^n-1)$.
%When $z$ satisfies $z^n=(-1)^{n+1}$, the isomorphism type of  $\al_{(n,\chi,z)}$
%is actually independent of the choice of $z$ and in this case
%we will write $\al_{(n,\chi)}=\al_{(n,\chi,z)}$.
%Note that $\al_{(n,\chi)}$ defines an $\SL(n,\CC)$--representation and that  the isomorphism type of this representation is independent of the choice of $z$.
%\item If $z=1$, then we write $\be_{(n,\chi)}=\al_{(n,\chi,1)}$.
%\end{enumerate}
In our notation we will not normally
distinguish between metabelian representations
of $\pi_1(N_K)$ and representations of $ \ZZ \ltimes H$.

In the following we say that a character $\chi \co H\to \CC^*$ has \emph{order $n$} if
it factors through $H/(t^n-1)$, but not through $H/(t^\ell-1)$ for any $\ell < n$.
 Given a character  $\chi \co H\to \CC^*$, let $t^i\chi$ be the character defined by $(t^i\chi)(h)=\chi(t^ih)$.
Any character $\chi \co H\to \CC^*$ which factors through $H/(t^n-1)$ must have order $k$ for some
divisor $k$ of $n$.
The following is a combination of \cite[Lemma~2.2]{BF08} and  \cite[Theorem~3.3]{BF08}.

\begin{thm} \label{thm2}
Suppose  $\chi \co H  \to \CC^*$ is a character that factors through $H/(t^n-1)$.
\begin{enumerate}
\item[(i)] $\al_{(n,\chi)}\co \ZZ\ltimes H\to \SL(n,\CC)$ is irreducible if and only if the character $\chi$ has order $n$.
%\item[(ii)] If $\chi \co H\to \CC^*$ is a character of order $n$, then $\al_{(n,\chi)}$
%defines an irreducible $\SL(n,\CC)$ representation.
%SF
\item[(ii)] Given two characters $\chi,\chi'\co H\to \CC^*$ of order $n$, the representations
$\al_{(n,\chi)}$ and $\al_{(n,\chi')}$ are conjugate if and only if $\chi=t^k\chi'$ for some $k$.
\item[(iii)] For any irreducible representation  $\al \co\ZZ\ltimes H\to \SL(n,\CC)$
 there exists a character $\chi \co H\to \CC^*$ of order $n$ such that
$\al$ is conjugate to $\al_{(n,\chi)}$.
\end{enumerate}
\end{thm}

\section{Main results}
\subsection{Metabelian characters as fixed points}

Set $\om = e^{2 \pi i/n}$ and recall the action
of the cyclic group $\ZZ/n =\langle \si \mid \si^n=1\rangle$ on
 representations $\al \co \pi_1(N_K) \to \SL(n,\CC)$ obtained by
  setting $(\si \cdot \al)(g)= \om^{\ep(g)} \al(g)$ for all $g \in \pi_1(N_K),$
where $\ep \co \pi_1(N_K) \to H_1(N_K)=\ZZ$.

We begin with the following  lemma.

\begin{lemma}\label{lem3}
Suppose $\al \co \pi_1(N_K) \to \SL(n,\CC)$ is a representation
whose associated character $\xi_\al \in X_n(N_K)$ is a fixed point
of the $\ZZ/n$ action. Then up to conjugation,
we have
\begin{equation} \label{eq2}
\al(\mu) =
\begin{pmatrix} 0& &\dots &z \\
z&0&\dots &0 \\
\vdots &\ddots &\ddots&\vdots \\
     0&\dots &z &0 \end{pmatrix},
     \end{equation}
for some (in fact any) $z \in U(1)$ such that $z^{n}=(-1)^{n+1}.$\end{lemma}

\begin{proof}
Let $c(t) = \det (\al(\mu)-tI)$ denote the characteristic polynomial
of $\al(\mu),$ which we can write as
$$c(t) = (-1)^n t^n + c_{n-1} t^{n-1} + \cdots + c_1 t +1.$$
Note that $c(t)$ is determined by the character $\xi_\al \in X_n(N_K)$,
and so assuming $\xi_\al$ is a fixed point of the $\ZZ/n$ action, we conclude that
$\al(\mu)$ and $\om^k \al(\mu)$ have the same characteristic polynomials for all $k$.
In particular,
\begin{eqnarray*}
c(t)&=&\det(\al(\mu)-tI)\\
&=&\det (\om^{-1} \al(\mu)-tI) \\
&=& \det (\om^{-1} \al(\mu)-(\om^{-1} \om)tI)\\
&=&\det (\om^{-1} I) \det (\al(\mu)- \om tI) \\
&=&  \det (\al(\mu)-t \om I) = c(\om t).
\end{eqnarray*}
However, $\om^k \neq 1$ unless $n | k$, and this implies
$0=c_{n-1}=c_{n-2} = \cdots = c_1$ and $c(t) = (-1)^n t^n +1.$
In particular the matrix $\al(\mu)$ and the matrix appearing in
Equation (\ref{eq2}) have the same set of $n$ distinct eigenvalues. This implies that the two matrices are conjugate.
\end{proof}

In order to prove Theorem \ref{thm1}, we establish the following more
general result.
\begin{thm} \label{thm4}
The fixed point set  of the $\ZZ/n$ action on $X_n(N_K)$
consists of characters $\xi_\al$ of the metabelian representations
$\al = \al_{(n,\chi)}$ described in Section \ref{sec2-3}. In other words,
$$X_n(N_K)^{\ZZ/n} = \{ \xi_\al \mid \al =\al_{(n,\chi)} \text{ for } \chi \co H_1(N_K; \ZZ[t^{\pm 1}]) \to \CC^*\}.$$
\end{thm}

Notice that Theorem~\ref{thm1} can be viewed as the
special case of Theorem \ref{thm4} where
$\al$ is irreducible. (Recall that irreducible representations are conjugate if and only if they define the same character.)
Notice further that not every reducible metabelian
representation is of the form $\al_{(n,\chi)}$.

\begin{proof}

We first show that if $\al \co \pi_1(N_K) \to \SL(n,\CC)$ is given as $\al = \al_{(n,\chi)},$
then $\si \cdot \al$ is conjugate to $\al$. This of course implies that $\xi_\al = \xi_{\si \cdot \al}$.

Assume then that $\al = \al_{(n,\chi)}.$
Then we have $$\al(\mu) =\begin{pmatrix}
0& \dots& &z \\
z& 0&\dots&0 \\
\vdots &\ddots &\ddots&\vdots \\
     0&\dots &z &0 \end{pmatrix},$$
     where $z$ satisfies $z^{n}= (-1)^{n+1}.$
Further, $\al(g)$ is diagonal for all $g\in [\pi_1(N_K), \pi_1(N_K)]$.
By definition of $\si \cdot \al,$ we  see that
$$(\si\cdot\al) (\mu) = \om \al(\mu)=
\begin{pmatrix} 0& \dots &&\om z \\
\om z &0&\dots &0 \\
\vdots &\ddots &\ddots&\vdots \\
     0&\dots &\om z &0 \end{pmatrix}$$
     and that $(\si \cdot \al)(g) = \al(g)$
     for all $g \in [\pi_1(N_K), \pi_1(N_K)]$.
It follows easily  from Theorem \ref{thm2} (2)
that $\si \cdot \al$ and $\al_{(n,\chi)}$ are conjugate;
however it is easy to see this directly too.
Simply take   $$P = \begin{pmatrix} 1&&& 0 \\
 &\om \\
&&\ddots  \\
     0&&&\om^{n-1} \end{pmatrix},$$
     and compute that
     $\si \cdot \al =P \al P^{-1}$ as claimed.

We now show the other implication, namely that each point
$\xi \in X_n(N_K)^{\ZZ/n}$ in the fixed point set
can be represented as the character $\xi=\xi_\al$
of a metabelian representation $\al = \al_{(n,\chi)},$
where $\chi \co H_1(N_K; \ZZ[t^{\pm 1}]) \to \CC^*$
is a character that factors through $H_1(N_K; \ZZ[t^{\pm 1}])/(t^n-1)$,
hence has order $k$ for some $k$ dividing $n$.
(Note that Theorem \ref{thm2} (1) tells us that
$\al_{(n,\chi)}$ is irreducible if and only
if $\chi$ has order $n$.)

By the general results on representation spaces and character varieties (see \cite{LM85}),
it follows that every point in the character variety $X_n(N_K)$ can be represented
as $\xi_\al$ for some semisimple representation $\al \co \pi_1(N_K) \to \SL(n,\CC)$.
Further, two semisimple representations $\al_1$ and $\al_2$
determine the same character if and only if $\al_1$ is conjugate to $\al_2.$
(This is evident from the fact that the orbits of the semisimple representations under conjugation
are closed.)

Given $\xi \in X_n(N_K)^{\ZZ/n}$,
we can therefore suppose that $\xi=\xi_\al$ for some semisimple representation $\al$.
Clearly $\si \cdot \al$ is also semisimple, and
since $\xi_\al = \xi_{\si \cdot \al}$, we conclude from the above that $\al$ and $\si \cdot \al$
are conjugate representations. This means  that there exists a matrix $A  \in \SL(n,\CC)$
such that $A \al A^{-1} = \si \cdot \al$,
in other words, for all $g\in \pi_1(N_K),$ we have
\begin{equation} \label{eq2} A \al(g) A^{-1} = \om^{\ep(g)} \al(g).\end{equation}
Lemma \ref{lem3} implies
$\al(\mu)$ is conjugate to  the matrix in Equation (\ref{eq2}).
It is convenient to conjugate $\al$ so that $\al(\mu)$ is diagonal,
meaning that
$$\al(\mu)= \begin{pmatrix} z&&& 0 \\
 &\om z  \\
&&\ddots  \\
     0&&&\om^{n-1} z \end{pmatrix},
$$
where $z$ satisfies $z^n=(-1)^{n+1}$.

We now apply (\ref{eq2}) to the meridian to conclude that
$$A \al(\mu) = \om \al(\mu) A,$$
 which implies
$A = (a_{ij})$ satisfies $a_{ij}=0$ unless $j = i+1 \mod(n).$
Thus, we see that
$$A = \begin{pmatrix} 0&\la_1&0& \dots &0 \\
0&0&\la_2&\dots &0 \\
 \vdots& \vdots & \ddots& \ddots &\vdots \\
 0&0&\dots&0&\la_{n-1}\\
 \la_n & 0 & \dots &0 &0
\end{pmatrix}$$
for some $\la_1, \ldots, \la_n$ satisfying $\la_1 \cdots \la_n = (-1)^{n+1}.$

It is completely straightforward to see that the characteristic polynomial of
$A$ is given by
$$\det (A-tI) =
(-1)^n(t^n - (-1)^{n+1}).$$
From this, we conclude that $A$ has as its eigenvalues the
$n$ distinct $n$--th roots of $(-1)^{n+1}.$ In particular, the subset of $\SL(n,\CC)$
of matrices that commute with $A$ is just a copy of the unique maximal torus
$T_A \cong (\CC^*)^{n-1}$ containing $A$.

For any $g \in [\pi_1(N_K), \pi_1(N_K)]$, we have
$\al(g) = (\si \cdot \al)(g)$. Thus it follows that
$A \al(g) A^{-1} = \al(g),$ and this implies that $\al(g) \in T_A$ for
all $g \in [\pi_1(N_K), \pi_1(N_K)]$. This shows that
the restriction of $\al$ to the commutator subgroup
$[\pi_1(N_K), \pi_1(N_K)]$ is abelian, and we conclude from this that
$\al$ is indeed metabelian. Notice that this, and an application of  Theorem \ref{thm2} (3),
completes the proof in the case $\al$ is irreducible.

In the general case, it follows from the discussion in Section  \ref{section:metabk} that
$\al$ factors through $\ZZ\ltimes H_1(N_K;\ZZ[t^{\pm 1}])$.
Let $H =  H_1(N_K;\ZZ[t^{\pm 1}])$.
Given a character $\chi \co H\to\CC^*$ we define the associated weight space $V_\chi$ by setting
$$V_{\chi}=\{ v\in \CC^n \, |\, \chi(h)\cdot v=\al(h)v \text{ for all }h\in H\}.$$
Recall that  $A\cdot \al(h)\cdot A^{-1}=\al(h)$ for any $h\in H$. It is straightforward so show that $A$ restricts to an automorphism of $V_\chi$.
Since $H$ is abelian there exists at least one character $\chi\co H\to \CC^*$ such that
$V_{\chi}$ is non--trivial. For any $i$ we denote by $t^i\chi$ the character given by $(t^i\chi)(h)=\chi(t^i h), h\in H$.

Note that $A$ has $n$ distinct eigenvalues and therefore is diagonalizable.
Since $A$ restricts to an automorphism of $V_\chi$, there is an
eigenvector $v$ of $A$ which lies in $V_{\chi}$. Let $\la$ be the corresponding eigenvalue. By the proof of \cite[Theorem~2.3]{BF08},
the map $\al(\mu)$ induces an isomorphism $V_{\chi}\to V_{t\chi}$.
We now calculate
\[ A\cdot \al(\mu)  v=(A\al(\mu)A^{-1})\cdot Av=\om \al(\mu) \cdot \la v=\la \om \cdot \al(\mu)v,\]
i.e.  $\al(\mu) v \in V_{t\chi}$ is an eigenvector of $A$ with eigenvalue $\om \la$.

Iterating this argument, we see that $\al(\mu)^iv$ lies in $V_{t^i\chi}$ and is an eigenvector of $A$ with eigenvalue $\om^i \la$.
Since $\om$ is a primitive $n$--th root of unity, the eigenvalues $\la,\om \la,\dots,\om^{n-1}\la$ are all distinct, and this implies that
the corresponding eigenvectors $v, \al(\mu)v, \dots, \al(\mu)^{n-1}v$  form a basis for $\CC^n$.

Let $m$ be the order of $\chi$, i.e. $m$ is the minimal number such that $\chi=t^m\chi$.
By the above we see that $\CC^n$ is generated by $V_{\chi}, V_{t \chi}, \dots,V_{t^m\chi}$.
Since the characters $\chi, t\chi, \dots,t^m\chi$ are pairwise distinct, it follows  that  $\CC^n$ is given as the direct sum $V_{\chi}\oplus V_{t\chi}\oplus \dots \oplus V_{t^{m-1}\chi}$.

We write $k=\dim_{\CC}(V_{\chi})$ and note that $n=km$. We note further that $\al(\mu)^m$ has   eigenvalues given by the set
\begin{equation}\label{eq3}
\{z^m,z^me^{2\pi i/k},\dots,z^me^{2\pi i(k-1)/k}\},
\end{equation}
and each eigenvalue has multiplicity $m$.
Clearly $\al(\mu)^m$ restricts to an automorphism of $V_{t^i\chi}$ for $i=0,\dots,m -1$, and
equally clearly we see that the restrictions all give conjugate representations.
This implies that the restriction of $\al(\mu)^m$ to $V_{\chi}$ has  eigenvalues in the
set (\ref{eq3}) above, each occurring with multiplicity $1$.
In particular we can find a basis
 $\{v_1,\dots,v_k\}$  for $V_\chi$ in which the matrix of $\al(\mu)^m$ has the form
 $$\al(\mu^m) =\begin{pmatrix}
0& \dots& &z^m \\
z^m& 0&\dots&0 \\
\vdots &\ddots &\ddots&\vdots \\
     0&\dots &z^m &0 \end{pmatrix}.$$
It is now straightforward to verify that with respect to the ordered basis
$$\left\{ \begin{array}{cccc}
v_1,&z^{-1}\al(\mu)v_1,&\dots,&z^{-(m-1)}\al(\mu)^{m-1}v_1,\\
v_2,&z^{-1}\al(\mu)v_2,&\dots,&z^{-(m-1)}\al(\mu)^{m-1}v_2,\\
\vdots &\vdots &\dots& \vdots \\
v_k,&z^{-1}\al(\mu)w_k,&\dots,&z^{-(m-1)}\al(\mu)^{m-1}v_k\end{array}\right\} ,$$
$\al$ is given by $\al(n,\chi)$.
\end{proof}

%=======================
\subsection{Application to twisted Alexander polynomials}\label{section:twialex}
As an application, we now prove the following result regarding twisted Alexander polynomials of knots
corresponding to metabelian representations.
In the following, we use  $\Delta^\al_{K,i}(t)$ to denote the $i$--th twisted Alexander
polynomial for a given representation $\al \co \pi_1(N_K) \to \SL(n,\CC)$ as presented
in \cite{FV09}.

\begin{prop} \label{prop5}
Let $\al$ be  a metabelian  representation of the form
$\al=\al_{(n,\chi)} \co \pi_1(N_K) \to \SL(n,\CC)$. Then
\[  \Delta^\al_{K,0}(t) =\left\{ \begin{array}{rl} 1-t^n, &\mbox{ if $\chi$ is trivial}, \\ 1,& \mbox{ otherwise.}\end{array} \right.\]
Furthermore
the twisted Alexander polynomial
 $\Delta^\al_{K,1}(t)$ is actually a polynomial
in $t^n.$
\end{prop}

\begin{remark}
In their paper \cite{HKL08},  C. Herald, P. Kirk, and C. Livingston
prove the same result using an entirely different approach (cf. p.~10 of \cite{HKL08}).
We also point out that Proposition \ref{prop5} gives a positive answer to Conjecture A from a
recent paper by M. Hirasawa and K. Murasugi (see  \cite{HM09}).
\end{remark}

\begin{proof}
The proof of the first statement is not difficult. It is immediate
when $\chi$ is trivial, and it follows by a direct calculation when
 $\chi$ is non--trivial.
%or from \cite{FJV09}.

We now turn to the proof of the second statement.
For $\th \in U(1)$ and any representation
$\be:�\pi_1(N_K) \to \GL(n,\CC)$, define the $\th$-twist of $\be$
to be the representation sending $g \in \pi_1(N_K)$ to
$\th^{\ep(g)} \be(g),$ where $\ep\co \pi_1(N_K) \to \ZZ$ is
 determined by the orientation of $K$.
We denote the newly obtained representation
by $\be_\th \co \pi_1(N_K) \to \GL(n,\CC)$.
Note that in case $\al:�\pi_1(N_K) \to \SL(n,\CC)$ and
$\th= e^{2 \pi ik/n}$ is an $n$-th root of unity,
$\al_\th$ is again an $\SL(n,\CC)$ representation.
The proof of the proposition relies on the
formula
\begin{equation} \label{eq4}
\Delta_{K,1}^{\be_\th}(t) = \Delta_{K,1}^{\be}(\th t).
\end{equation}
This formula is well-known and follows directly from the definition of the twisted Alexander
polynomial.
Equation (\ref{eq4}) combines with Theorem \ref{thm1} to complete the proof, as we now explain.
Take $\om = e^{2 \pi i/n}$. If $\al=\al_{(n,\chi)}$ is metabelian,
then Theorem \ref{thm1} shows that its conjugacy class is fixed under the $\ZZ/n$
action. In particular, since $\al$ and $\al_\om$ are
conjugate, Equation (\ref{eq4}) shows that
$$\Delta^{\al}_{K,1}(t) =�\Delta^{\al_\om}_{K,1}(t)=�\Delta^{\al}_{K,1}(\om t).$$
Expanding
$\Delta^{\al}_{K,1}(t)= \sum a_i t^i$
and using the fact that $t^k = (\om t)^k$ if and only if $k$ is a multiple
of $n$,
this shows that $a_k =0$ unless $k$ is a
multiple of $n$ and this completes the proof.
\end{proof}

\end{document}